\documentclass[12pt,twoside]{article}

\usepackage{setspace}
\usepackage{amssymb}
\usepackage{graphicx}
\usepackage{color}
\usepackage{amsmath}
\usepackage{amssymb}
\usepackage{amsthm}
\usepackage{graphicx,color}
\usepackage{pstricks}
\usepackage{pst-tree}
\usepackage{multido}
\def\R{{\mathbb R}}

\def\N{{\mathbb N}}

\def\C{{\mathbb C}}

\newtheorem{thm}{Theora}[section]
\newtheorem{theo}[thm]{Theorem}

\newtheorem{lem}[thm]{Lemma}
\newtheorem{prop}[thm]{Proposition}

\newtheorem{rem}[thm]{Remark}

\newcommand\dint{\displaystyle\int}
\newcommand\ds{\displaystyle\sum}

\newcommand\dprod{\displaystyle\prod}

\newcommand\dsup{\displaystyle\sup}

\newcommand\dlim{\displaystyle\lim}

\begin{document}
\pagestyle{myheadings} \markboth{A. Sergio and B. Smii}{Asymptotic
expansion for SDE's }
 \title{Asymptotic expansions for SDE's with small multiplicative noise.}
  \def\lhead{S.\ ALBEVERIO, B. Smii, } 
  \def\rhead{Asymptotic expansions for SDE's with small multiplicative noise.} 

\author{Sergio Albeverio${}^\sharp{}^\flat$ and  Boubaker Smii${}^\flat $  }
\maketitle {\small

 \noindent ${}^\flat$: King Fahd University of
Petroleum and Minerals, Dept. Math. and Stat., Dhahran 31261, Saudi
Arabia.\,\,{\tt boubaker@kfupm.edu.sa}

 \noindent  ${}^\sharp$: Dept. Appl. Mathematics, University of Bonn, HCM, BiBoS, IZKS. \,\,{\tt albeverio@uni-bonn.de}

\maketitle

\begin{abstract}
Asymptotic expansions are derived as power series in a small
coefficient entering a nonlinear multiplicative noise and a
deterministic driving term in a nonlinear evolution equation.
Detailed estimates on remainders are provided.
\end{abstract}
\section{Introduction}
A description of the evolution of dynamical systems of concern in
disciplines like physics, biology, geology, engineering, economics
in terms of differential equations is appropriate. Sometimes it is
natural to investigate to which extent an external small
perturbation (forcing) can change the deterministic evolution. This
can be discussed in the sense of asymptotic expansions in powers of
a small parameter in front of the perturbation.\\This problem has
been studied in particular for the case where the perturbation is an
additive noise of the Brownian or L\'evy type. In the case of
evolution in a Hilbert space with global Lipschitz coefficient this
has been discussed for the Brownian additive noise case in
\cite{Marc2}. For stochastic partial differential equations, related
to evolutions on a Hilbert resp. Banach space, this has been
discussed with non globally Lipschitz coefficients in a situation of
dissipativity in \cite{AlDiPMa} for the case where the additive
noise is given by Brownian motion, and in \cite{ALEBOU} for the case
of additive L\'evy type noise. For related work determining the
invariant measures in such cases, see, \cite{AlDiPMa} resp.
\cite{ALDIBOU}, \cite{Ma}, see also \cite{ADPM}.\\In the present
work we consider the finite dimensional case with multiplicative
noise of Gaussian or L\'evy type. Even for this relatively simple
case to the best of our knowledge rigorous mathematical results seem
quite scarce, despite the conceptual importance of the problem in
relation, e.g., to classical mechanics. Of course the problem can be
looked upon, in principle, as a particular case of " perturbation
theory". On the other hand often rigorous "perturbation theory" is
either limited to (general) linear systems and associated semi
groups in Hilbert spaces, where a rich mathematical theory has been
developed, particularly in connection with spectral problems in
quantum theory(see, eg., \cite{EFR} \cite{Kato}, \cite{Maslov},
\cite{Reed}), or else to
particular nonlinear cases (see, eg., \cite{Costin, Eck, Sand}). For further motivations, mainly from physics, engineerings and mathematical finance,
see, e.g., \cite{CG}, \cite{LUT}, \cite{UY}.\\
 A classical area where asymptotic perturbation methods originally arise is classical celestial
mechanics( since the work by, S. Laplace, S. Poisson, C.F. Gauss, H.
Poincar\'e). Here non linear and singularity effects are essential
and particular methods have been developed, see, eg., \cite{ARNOL},
\cite{BIRKHOFF}. These are also related to perturbation theory
around the solutions of the classical motion of the harmonic
oscillators, see, eg., \cite{BOGO}, \cite{BOG}, \cite{Sand}. The
stochastic case of
Hamiltonian systems is studied in \cite{Vas}, \cite{Wen}.\\
Perturbation theory in infinite dimensional systems has been studied
in connection with hydrodynamics(small viscosity expansions, see.,
\cite{AFER}, \cite{Costin}  small time expansions \cite{Costin}),
quantum field theory \cite{AGoWu}, \cite{AlGYo}, \cite{EFR},
\cite{GST}, \cite{GS}, \cite{GOS}, \cite{BS}, neurobiological
systems \cite{AlDiP}, \cite{AlDiPMa}, \cite{ALEBOU},
 \cite{AHMA}, \cite{MM}, \cite{Tu2} .\\
 Let us also briefly mention connections with Laplace and stationary
 phase methods, see,eg. \cite{A}, \cite{AHMA}, \cite{AL}, \cite{ARoSk}, \cite{ASK}, \cite{Bre},
 \cite{Tu2}.\\
The present paper considers deterministic, resp. stochastic finite
dimensional differential equations, which are first order in time,
and have smooth  coefficients satsfying growth restrictions. The
driving multiplicative forcing term resp. noise is of the general
L\'evy type. An asymptotic expansion in powers of a small parameter
on which the diffusion
coefficient depends is exhibited with good detailed control on remainders.\\
Some possible applications  are mentioned at the end of this paper.\\
The structure of this paper is as follows:\\
In section 2 we present the concrete small noise expansion (assumed
to exist) of the original stochastic differential equation, in terms
of solutions of linear random differential equations, assuming that
solutions exist and are unique.\\In section 3 we discuss existence
and uniqueness of solutions of the original SDE.In section 4 we
discuss the solutions of the random equations for
the expansion coefficients.\\
In section 5 we prove the asymptotic character of the expansion.
Section 6 is reserved to some comments on applications.

 \setcounter{section}{1}
\section{Small noise expansion of an evolution equation}

We consider the evolution equation:

\begin{equation}
    \left \{
        \begin{aligned}
            u(t) &= u(0) + \int_0^t \beta \left( u(s) \right) \mathrm{d}s + \int_0^t \sigma_{\varepsilon} \left( u(s) \right) \eta (\mathrm{d}s) \\
            u(0) &= u^0 \text{,} \ u(t) \in \mathbb{R}^d \text{,} \ t \in [0,\infty)  \text{,}  \ \varepsilon > 0 \text{.}
        \end{aligned}
    \right.
    \label{eqn:1}
\end{equation}
$\beta(.):
\R^d\,\longrightarrow\,\R,\,\,\sigma_{\varepsilon}(.):\R^d\longrightarrow\,\R^{d\times
d}$ are measurable functions resp. $d\times d$ matrix functions
satisfying some additional assumptions, (e.g, globally Lipschitz and
growth conditions(at infinity)).

 $\eta$ can be a signed bounded variation measure (in
which case the integral is understood as a Stieltjes integral) or
the heuristic derivative of a L\'{e}vy process in $\R^d$ (in which
case the integral should be interpreted as a stochastic integral).
For simplicity of notations we use the unified notation $\eta(ds)$
e.g. if $\eta=B$ in $\R^d$ then $\eta(ds)=dB(s).$ Moreover if $\eta$
has a jump component $u(s)$ in (\ref{eqn:1}) should be understood as
$u(s^-)$. We hope the meaning is always clear from the context in which we operate. See section \ref{sec:1} for precise assumptions. \\
Our purpose is to show that under certain hypothesis on
$\beta,\sigma_{\varepsilon },\eta$ the solution $u=u_{\varepsilon}$
of the evolution equation (\ref{eqn:1}), assumed to exist, can be
written as:

\begin{equation}\label{rah}
    u_{\varepsilon}(t) = u_0(t) + \varepsilon u_1(t) + \ldots + \varepsilon^m u_m(t) + R_m(t,\varepsilon) \text{,}
\end{equation}

with $u_i:\,[0,\,\infty)\,\longrightarrow\,\R^d$ measurable and
$\parallel R_m(t,\varepsilon) \parallel \le C_m(t)
\varepsilon^{m+1}$, for all $m \in \mathbb{N}$ and $\varepsilon > 0$
sufficiently small, for some $C_m(t) > 0$
independent of $\varepsilon$. Here $\parallel. \parallel$ denotes the norm in $\R^d.$\\
To obtain the desired expansion we shall assume that there are
Taylor expansions of $\beta(x)$ and $\sigma_{\varepsilon}(x)$ in
their variable $x\in\,\R^d$ and, moreover, $\varepsilon \, \longrightarrow\, \sigma_{\varepsilon }(x)$ is ${\cal C}^{M+1},\,M\,\in\,\N,$  for every fixed $x\,\in\,\R^d$.\\
Let us first introduce some useful notations: \\
For $\alpha = (\alpha_1, \ldots, \alpha_d) \in \mathbb{N}^d_0$ (with
$\N_0=\{0\}\cup\,\N,\,\,\N=\{1,2,....\}$), and $x = (x_1, \ldots,
x_d) \in \mathbb{R}^d $, we define:
\begin{itemize}
    \item The length of $\alpha $ by $\left| \alpha \right| = \alpha _1 + \alpha _2 + \cdots + \alpha_d$.
    \item $\alpha ! =: \alpha_1 ! \alpha_2 ! \cdots \alpha_d !$
    \item $x^{\alpha} := x_1^{\alpha_1} x_2^{\alpha_2} \cdots x_d^{\alpha_d}$
\end{itemize}

 The derivative of a function $f$ of order $\left| \alpha \right|\,\in\,\N_0$ is defined by:

\begin{equation}
    f^{(\alpha)}=D^{\alpha}f = \frac{\partial^{\left| \alpha \right|}f}{\partial x_1^{\alpha_1} \partial x_2^{\alpha_2} \cdots
    \partial x_d^{\alpha_d}},\,\,\,\,\,\,\,\,\,D^{0}f=f.
\end{equation}

We have the following lemma:

\begin{lem}\label{l1}
    Let $f$ be a complex-valued function in $C^{p+1} \left( B(x_0,r) \right)$, $r>0$, $x_0 \in \mathbb{R}^n$, for some $p \in \mathbb{N}_0$, $n \in \mathbb{N}$, where

    $B(x_0, r)$ is the open ball in $\R^d,\,\,d\,\in\,\N,$ of center $x_0$ and radius $r$.\\Then for
    any $x \in B(x_0, r)$ we have Taylor's expansion formula:
    \begin{equation}\label{haj}
        f(x) = \sum_{\left| \alpha \right| \le p} \frac{D^{\alpha}f(x_0)}{\alpha !} (x-x_0)^{\alpha} + R_p \left( f^{(p+1)} (x_0,x)
        \right),
    \end{equation}
    with $D^{\alpha}f(x_0)$ the evaluation of $D^{\alpha}f$ at $x=x_0$ and\\
    $ R_p \left( f^{(p+1)} (x_0,x) \right) = \frac{(p+1)(x-x_0)^{\alpha}}{\alpha !}  \ds_{\left| \alpha \right| = p+1} \left( \int_0^1 (1-s)^p D^{\alpha}
    f \left( x_0 + s\,(x-x_0) \right) \mathrm{d}s \right)$. \\
    Moreover, setting \\
    $C_p (x_0,x) = \frac{p+1}{\alpha !} \ds_{\left| \alpha \right| = p+1} \left( \int_0^1 (1-s)^p \| D^{\alpha} f \left( x_0 + s(x-x_0) \right)
     \| \mathrm{d}s \right)$, we have the bound
    \begin{equation}
        \mid R_p \left( f^{(p+1)} (x_0,x) \right) \mid \le C_p (x_0,x) \| x-x_0 \|^{\alpha} \text{.}
    \end{equation}
\end{lem}
\begin{proof}
    See, e.g. \cite{MAG}.
\end{proof}

Using the previous lemma, we have for any $m \in \mathbb{N}_0$,
$\varepsilon \geq 0$, the following:

\begin{prop}\label{NA}
    If, for any $\varepsilon\,\geq 0, u\,\in\,{\cal C}^{N+1}(\R^d),$
    for some $d\,\in\,\N,\,N\,\in\,\N_0,$ and
    \begin{equation}
        u(\varepsilon) = u_0 + \varepsilon u_1 + \varepsilon^2 u_2 + \cdots + \varepsilon^N u_N + R_N^u (\varepsilon) \text{,}
        \label{eqn:2}
    \end{equation}
    with $u_i \in \mathbb{R}^d$, independent of $\varepsilon$, and $\left| R_N^u (\varepsilon) \right| \le C_N^u \varepsilon^{N+1}$ for some
     $C_N^u > 0$, then for any $u_0 \in \mathbb{R}^d$, $f \,\in\,C^{p+1}(\R^d),$ with  $p \in
     \mathbb{N}_0,$ we have:
    \begin{align}\label{abd}
        f \left( u(\varepsilon) \right)&=& \hspace*{-5cm}\ds_{\left| \alpha \right| \le p}
        \frac{D^{\alpha} f\left( u_0 \right)}{\alpha !} \left( u( \varepsilon) - u_0 \right)^{\alpha} +
        R_p^{u( \varepsilon)} \left( f^{(p+1)} \left( u_0,u(\varepsilon) \right) \right) \nonumber \\
        &=& \sum_{\left| \alpha \right| \le p} \frac{D^{\alpha} f\left( u_0 \right)}{\alpha !} \left(\sum_{l=1}^{N} \varepsilon^l u_l +
         R_N^u(\varepsilon) \right)^{\alpha} + R_p^{u(\varepsilon)} \left( f^{(p+1)} \left( u_0,u(\varepsilon) \right) \right) \text{.}
    \end{align}
\end{prop}
\begin{proof}
This is immediate from lemma \ref{l1}, with $x\,\in\,\R^d$ replaced
by $u(\varepsilon)\,\in\,\R^d$ and $x_0$ replaced by
$u_0\,\in\,\R^d,$ and denoting the remainder $R_p$ in (\ref{haj}) by
$R_p^{u(\varepsilon)}$, to recall that it refers to the function
$f(u(\varepsilon))$ instead of $f(x)$.
\end{proof}
\begin{rem}
We point out that the remainder $R_{N}^u(\varepsilon)$ in
(\ref{eqn:2}), referring to the asymptotic expansion of $u$, should
not be confused with the remainder $R_p^{u(\varepsilon)}$
in(\ref{abd}), for $p=N$ (which refers to the function f(u)).
\end{rem}
Now for any $N \in \mathbb{N}_0$ we have, by the definition of
$x^{\alpha}$ given before and the binomial formula for the powers of
the components $y_i$ of the vector $y=\sum_{l=1}^{N} \varepsilon^l
u_l + R_N^u( \varepsilon)$ in $\R^d$ on the l.h.s of (\ref{abd1}),
to be taken to the multi power $\alpha,$ i.e.
$y^{\alpha}=\dprod_{i=1}^d\,y_i^{\alpha_i}$.

\begin{align}\label{abd1}
    \left( \sum_{l=1}^{N} \varepsilon^l u_l + R_N^u( \varepsilon) \right)^{\alpha} = \dprod_{i=1}^d\,\Big[\sum_{\substack{\alpha_{1,i},\ldots,\alpha_{{N+1},i}=0 \\
    \alpha_{1,i}+\cdots+\alpha_{{N+1},i} =  \alpha_i }}^{ \alpha_i } \frac{ \alpha_i  !}{ \alpha_{1,i}! \cdots \alpha_{{N+1},i}!}
    \varepsilon^{\alpha_{1,i} + 2\alpha_{2,i} + \cdots + N\alpha_{N,i}} u_{1,i}^{\alpha_{1,i}} u_{2,i}^{\alpha_{2,i}} \cdots u_{N,i}^{\alpha_{N,i}}\nonumber\\ \left( R_{N,i}^u(\varepsilon) \right)^{\alpha_{{N+1},i}}
    \Big] \text{,}
\end{align}
$u_{j,i}$ is the $i-$th component of the vector $u_j,\,j=1,...,N;
i=1,...,d$, and $R_{N,i}^u(\varepsilon)$ is the $i-$th component of
the vector $R_{N}^u(\varepsilon)$.\\
Note that $\left( R_{N,i}^{u}( \varepsilon)
\right)^{\alpha_{{N+1,i}}}$ is bounded in norm by a positive
constant $\left( C_N^{u}\right)^{\alpha_{{N+1},i}}$ times
$\varepsilon^{(N+1)\alpha_{{N+1},i}}$ (since $\left|
R_N^u(\varepsilon) \right| \le C_N^u\, \varepsilon^{N+1}$, $C^u_N >
0$ from (\ref{eqn:2}) ). We also point out that
$\alpha_i\,\in\,\N_0,\,i=1,...,d$ are the components of
$\alpha\,\in\,\N_0^d,$ whereas the $\alpha_{j,i}\,\in\,\N_0
,\,j=1,...,N+1$ are restricted by the conditions appearing under the summation.\\
Thus we get, from equations (\ref{abd}), (\ref{abd1}) and by the
assumption in Proposition \ref{NA}, for $N\,\in\,\N_0,
p\,\in\,\N_0$:

\begin{eqnarray}\label{abd2}
    f \left( u(\varepsilon) \right) &=& \sum_{\left| \alpha \right| \le p} \frac{D^{\alpha} f\left( u_0 \right)}{\alpha !} \dprod_{i=1}^d\,\Big[\sum_{\substack{\alpha_{1,i},\ldots,\alpha_{{N+1},i}=0 \\
    \alpha_{1,i}+\cdots+\alpha_{{N+1},i} =  \alpha_i }}^{ \alpha_i } \frac{ \alpha_i  !}{ \alpha_{1,i}! \cdots \alpha_{{N+1},i}!}
    \varepsilon^{\alpha_{1,i} + 2\alpha_{2,i} + \cdots + N\alpha_{N,i}}\nonumber\\&& u_{1,i}^{\alpha_{1,i}} u_{2,i}^{\alpha_{2,i}} \cdots u_{N,i}^{\alpha_{N,i}} \left( R_{N,i}^u(\varepsilon) \right)^{\alpha_{{N+1},i}}
    \Big]   + R_p^{u(\varepsilon)} \left( f^{(p+1)} \left( u_0,u(\varepsilon)
     \right) \right),
\end{eqnarray}
with $\alpha= (\alpha_1,\cdot\cdot\cdot,\alpha_d)\,\in\,\N_0^d.$\\
We can rewrite (\ref{abd2}) by grouping the terms with the same
power $k$ of $\varepsilon, k\,\leq\,N.$ Denote the term of exact
order $k,\, 0\leq k \leq N, $ in $\varepsilon$ appearing on the
right hand side of (\ref{abd2}) by $[f \left( u(\varepsilon)
\right)]_k$. To compute it we first observe that we have to take
$\alpha_{N+1,i}=0,$ otherwise, due to the bound on $R_N$, the
effective presence of $\left( R_{N,i}^{u}( \varepsilon) \right) $
would give a term bounded by $\varepsilon^{N+1},$ with  $N+1\,>k.$
Then we have to take the sum over the $\alpha_{j,i}\,\in\,
\{0,1,...,\alpha_{i}\},\,j=1,...,N,\,i=1,...,d$ restricted by
$\alpha_i\,\,\in \N_0$ and satisfying:
\begin{enumerate}
\item $\ds_{i=1}^d\,\ds_{j=1}^N\,j\,\alpha_{j,i}=k$,
\item $\alpha_i=\ds_{j=1}^N\,\alpha_{j,i}$.
\end{enumerate}
For $k=0$ we must then have $\alpha_{1,i}=0$ for all $j,i$ and thus:
\begin{equation}\label{f0}[f \left( u(\varepsilon)
\right)]_0=f(u_0).
\end{equation}
For $k=1$ we have from (\ref{eqn:1}) that $\alpha_{1,i}=1$ for some
$i,$ all other $\alpha_{j,k}$ being $0.$ This implies $\alpha_i=1,
\alpha_l=0$ for all $l\neq i.$ Inserting this into (\ref{abd2}) we
get $[f \left( u(\varepsilon) \right)]_1= \sum_{i=1}^d\, {\partial
\over {\partial\,y_i}}\,f(y)_{\mid_{y=u_0}}\,u_{1,i}$. Introducing
the short notation $\partial_i\,f(u_0):={\partial \over
{\partial\,y_i}}\,f(y)_{\mid_{y=u_0}},$ we have thus:
\begin{equation}\label{f1}
[f \left( u(\varepsilon)
\right)]_1=\sum_{i=1}^d\,\partial_i\,f(u_0)\,u_{1,i}.
\end{equation}
For $k=2$ we have from (\ref{eqn:1}):
$\sum_{i=1}^d\,\sum_{j=1}^N\,j\,\alpha_{j,i}=\sum_{i=1}^d\,\sum_{j=1}^2\,j\,\alpha_{j,i}=2.$
This gives the possibility $j=2$ and $\alpha_{2,i}=1$ for some $i,
\alpha_{l,l'}= 0, \,l\neq\,2,\,l'=1,...,d.$ This provides the
contribution $ \sum_{i=1}^d\,\partial_i\,f(u_0)\,u_{2,i}$ to $[f
\left( u(\varepsilon) \right)]_2$. Another contribution is given by
the case $j=1$ and $\alpha_{1,i}=1, \alpha_{1,i'}=1$ for some $i,
i'=1, i\neq i'$, with $\alpha_{j,l}= 0,\,\forall\,j\neq 1,
\forall\,l$ and $\alpha_{1, m}=0, \forall m \neq i'.$ In this case
we get the contribution:
\begin{equation}\ds_{i=1}^d\,\ds_{i'=1\\ \atop i'
\neq i}^d\,{{\partial^2} \over {{\partial y_i} \, {\partial
y_{i'}}}}\,f(y)_{\mid_{y =u_0} }u_{1, i}\,u_{1,i'},\,\,\,to\,\,\,[f
\left( u(\varepsilon) \right)]_2.
\end{equation}
Denoting $\partial_i\partial_{i'}\,f(u_0):={{\partial^2} \over
{{\partial y_i} \, {\partial y_{i'}}}}\,f(y)_{\mid_{y =u_0} },$ we
have in total:
\begin{equation}\label{f2}[f \left( u(\varepsilon)
\right)]_2=\ds_{i=1}^d\,\partial_i\,f(u_0)\,u_{2,i}\,+\,{1 \over
{2!}}\,\ds_{i=1}^d\,\sum_{i'=1\\ \atop i' \neq
i}^d\,\partial_i\partial_{i'}\,f(u_0)\,u_{1, i}\,u_{1,i'}.
\end{equation}
In a similar way we can get the contribution $[f \left(
u(\varepsilon) \right)]_k$ of order $k\,\geq\,3$. It is easy to see
that it contains the term $u_{k, i}$ only linearly and that it
depends in a homogeneous way of total order $k$ in components of the
coefficients $u_{k-1}, u_{k-2},...,u_1, u_0.$ Thus introducing the
short notation:
\begin{equation}
\partial_1 \cdot\cdot\cdot \partial_k\,f(u_0):={\partial^k \over {\partial y_1 \cdot\cdot\cdot \partial
y_k}}\,f(y)_{\mid_{y =u_0} },
\end{equation}
we have
\begin{equation}\label{hiii}
[f \left( u(\varepsilon)
\right)]_k=\ds_{i=1}^d\,\partial_i\,f(u_0)\,u_{k,i}\,+\,{1 \over
{k!}}\,{\ds_{i_1,...,i_k=1}^d}^\star\,\partial_{i_1} \cdot\cdot\cdot
\partial_{i_k}\,f(u_0)\,u_{1, i_1}u_{1, i_2}\cdot\cdot\cdot u_{1,
i_k} + B_k^{f}(u_0,...,u_{k-1}),
\end{equation}
for some function $B_k^f$ which is a sum of monomials in components
of the variables $u_0,....,u_{k-1}$. The short notation
${\sum_{i_1,...,i_k =1}^d}^\star$ \,\,stands for the sum over all
$i_1,..., i_k,\,i_j\,\in\,\{1,...,d\},j=1,...,k,$ where all $i_j$
are
different.\\
By replacing simply $f$ by $\beta_l$ formula (\ref{hiii}) yields the
term of exact order $k$ in the asymptotic expansion in powers of
$\varepsilon$ of the $l-$th components of the coefficient
$\beta(u(\varepsilon))$ in (\ref{eqn:1}).\\\\
In order to get a corresponding expansion in powers of $\varepsilon$
for the matrix elements
$(\sigma_{\varepsilon})_{l,l'}(u(\varepsilon)))$,\\$l,l'=1,...,d$ of
the matrix $\sigma_{\varepsilon}$ in powers of $\varepsilon$ we have
to take care of the fact that, as opposite to $f$ and
$\beta_l,\,(\sigma_{\varepsilon})_{l,l'}$ also depends on
$\varepsilon,$ not only on its argument. \\Let us assume that:
\begin{equation}\label{abd3}
    \sigma_{\varepsilon} (\, x \,) = \sum_{i=0}^{M} \sigma_i(\, x \,) \varepsilon^i +
    R_M^{\sigma}(\varepsilon)(x),\,\,for\,any\,x\,\in\,\R^d.
\end{equation}

with $\dsup_{x\,\in\,\R^d}\parallel R_M^{\sigma} (\, x \, ,
\varepsilon) \parallel \le C_{M,\sigma} \varepsilon^{M+1}$,
($\parallel.\parallel$ denoting here the norm of the matrix
$R_M^{\sigma} (\, x \, , \varepsilon)$),\,$\sigma_i \,\,a\,\,d\times
d$ matrix and $C_{M,\sigma}
> 0$. 
Let us also assume that the elements
$(\sigma_i(x))_{l,l'},\,l,l'=1,...,d$ of the matrix $\sigma_i(x)$
belong to  ${\cal C}^{s+1}(\R^d)$ as functions of $x\,\in\,\R^d.$\\
For all $i=0,\ldots,M$, $M \in \mathbb{N}_0$ and for any $N \in
\mathbb{N}_0$, $s \in \mathbb{N}_0$ we have, from (\ref{rah}),
(\ref{abd3}), using the r.h.s of (\ref{abd2}):

\begin{align}\label{omra}
    \sigma_{\varepsilon} \left( u(\varepsilon) \right) &= \sum_{j=0}^{M} \sigma_j \left( u(\varepsilon) \right) \varepsilon^j + R_M^{\sigma} (\varepsilon) \nonumber\\
    &= \sum_{j=0}^{M} \sigma_j \left( \sum_{k=0}^{N} \varepsilon^k u_k + R_N^u (\varepsilon) \right) \varepsilon^j + R_M^{\sigma}(\varepsilon) \nonumber \\
    &= \sum_{j=0}^{M}\,\varepsilon^j\, \left[ \sum_{\left| \gamma \right| \le s} \frac{D^{\gamma}\sigma_j(u_0)}{\gamma !} \left( \sum_{k=0}^{N} \varepsilon^k u_k + R_N^u(\varepsilon) - u_0 \right)^{\gamma} + R_s^{\sigma_j} \right] + R_M^{\sigma}(\varepsilon) \nonumber \\
    &= \sum_{j=0}^{M}\,\varepsilon^j\, \left[\sum_{\left| \gamma \right| \le s} \frac{D^{\gamma}\sigma_j(u_0)}{\gamma !} \dprod_{i=1}^d \sum^{\gamma_i}_{\substack{\gamma_{1,i},\ldots,\gamma_{{N+1},i}=0 \\
    \gamma_{1,i}+\cdots+\gamma_{{N+1},i} =  \gamma_i }} \frac{\gamma_i !}{ \gamma_{1,i}! \gamma_{2,i}! \cdots \gamma_{{N+1},i}!} \varepsilon^{ \gamma_{1,i} + 2\gamma_{2,i} +
     \cdots + N\gamma_{N,i}} \right. \nonumber \\
    &\qquad \left. u_{1,i}^{\gamma_{1,i}} u_{2,i}^{\gamma_{2,i}} \cdots u_{N,i}^{\gamma_{N,i}} \left( R_N^{u(\varepsilon)}(\varepsilon) \right)^{\gamma_{{N+1},i}} + R_s^{\sigma_j} \right] +
    R_M^{\sigma}(\varepsilon)
\end{align}
Here $R_s^{\sigma_j} $ is a short notations for $R_s^{\sigma_j}
\Big(\ds_{k=0}^N\,\varepsilon^k\,u_k+\,R^{u}_{N}(\varepsilon)\Big)$.\\
Let us note that (\ref{omra}) is a relation between matrices, to be
understood elements by elements. $D^{\gamma}\sigma_i(u_0)$ has to be
interpreted as $D^{\gamma}$ applied to the elements
$(\sigma_i)_{l,l'},\,\,l,l'= 1,...,d$ of the
matrix $\sigma_i,$ evaluated then at $u_0.$\\
Proceeding as in the case of the expansions of $f$ and $\beta_l$ we
exhibit the coefficient $[\sigma_{\varepsilon} \left( u(\varepsilon)
\right)]_k$ of the power $k,\,0 \leq k \leq \min(M,N) $ in the
development of $\sigma_{\varepsilon, l,l'} \left( u(\varepsilon)
\right)$ on the right hand of (\ref{omra}). We shall write
$[\sigma_{\varepsilon} \left( u(\varepsilon) \right)]_k $ in matrix
form, but it should be understood element by element. As we did for
$f$ we have to set for this $\gamma_{{N+1},i} =0,\,i=1,...,d.$
Moreover we observe that (\ref{omra}) contains a sum of products of
$\varepsilon^j$ times a sum of terms with power $\sum_{i=1}^d\,
\left(\gamma_{1,i}+2\,\gamma_{2,i}+\,\cdot\cdot\cdot+N\,\gamma_{N,i}\right)$
in $\varepsilon$, hence the analogues of (\ref{eqn:1}),
(\ref{eqn:2}) we had for $[f \left( u(\varepsilon) \right)]_k$ are:
\begin{enumerate}
\item $j+
\ds_{i=1}^d\,\ds_{l=1}^N\,l\,{\gamma_{l,i}}=k,\,\,i=0,...,M$.
\item $\gamma_i=\ds_{l=1}^N\,{\gamma_{l,i}},$
with $\gamma_{l,i}=\{0,1,....,\gamma_i \}, l=1,...,N, i=1,...,d,
\gamma_i\,\in\,\N_0.$
\end{enumerate}

We see from 1. that we must have $j \leq k.$ Let us first compute
the terms for $k=0,1,2.$ We have:
\begin{equation}\label{s0}
[\sigma_{\varepsilon} \left( u(\varepsilon) \right)]_0=\sigma_{0}
\left( u_0 \right),
\end{equation}
 since $k=0$ implies $j=0, \gamma_{j,i}=0\,\,for\,\,all\,j,i.$ To obtain $[\sigma_{\varepsilon}
\left( u(\varepsilon) \right)]_1$ we observe that from 1 we have the
possibilities a) $j=0$ and $\gamma_{1,i}=1$ for some
$i,\,\gamma_{1,l}=0\,\,\forall\,l\neq\,i,\,\,
\gamma_{2,i}=...=\gamma_{N,i}=0, $ for all $d\,\in\,\N,$ or b)
$j=1,\,\gamma_{l,i}=0,\,\,\forall\,l,i.$ Thus we have:
\begin{equation}\label{s1}
[\sigma_{\varepsilon} \left( u(\varepsilon)
\right)]_1=\ds_{i=1}^d\,\partial_i\,\sigma_{0} \left( u_0
\right)\,u_{1,i}+\sigma_{1} \left( u_0 \right).
\end{equation}
For $k=2$, we have for $j=0$ only the possibilities we already
discussed for $[f \left( u(\varepsilon) \right)]_2$, so we get a
contribution
\begin{equation}\ds_{i=1}^d\,\partial_i\,\sigma_0(u_0)\,u_{2,i}\,+\,{1 \over
{2!}}\,\ds_{i,i'=1\\ \atop i' \neq
i}^d\,\partial_i\partial_{i'}\,\sigma_0(u_0)\,u_{1, i}\,u_{1,i'}.
\end{equation}
For $j=1$ we have the possibilities given by $\sum_{i=1}^d\,
\left(\gamma_{1,i}+2\,\gamma_{2,i}+\,\cdot\cdot\cdot+N\,\gamma_{N,i}\right)=1,$
which are those discussed for $[f \left( u(\varepsilon) \right)]_1$,
hence we get the contribution
$\sum_{i=1}^d\,\partial_i\,\sigma_1(u_0)\,u_{1,i}+\sigma_2(u_0)$. In
total then we get for any $l, l'=1,...,d:$
\begin{equation}\label{s2}[\sigma_{\varepsilon}\left( u(\varepsilon)
\right)]_2=\ds_{i=1}^d\,\partial_i\,\sigma_{0}(u_0)\,u_{2,i}\,+\,{1
\over {2!}}\,\ds_{i,i'=1\\ \atop i' \neq
i}^d\,\partial_i\partial_{i'}\,\sigma_{0}(u_0)\,u_{1,
i}\,u_{1,i'}+\ds_{i=1}^d\,\partial_i\,\sigma_1(u_0)\,u_{1,i}+\sigma_2(u_0).
\end{equation}
In the general case we see that
\begin{equation}\label{mama}
[\sigma_{\varepsilon} \left( u(\varepsilon)
\right)]_k=\ds_{i=1}^d\,\partial_i\,\sigma_{0}(u_0)\,u_{k,i}\,+\,{1
\over {k!}}\,{\ds_{i_1,...,i_k=1}^d}^\star\,\partial_{i_1}
\cdot\cdot\cdot
\partial_{i_k}\,\sigma_{0}(u_0)\,u_{1, i_1}u_{1, i_2}\cdot\cdot\cdot u_{1,
i_k}+\sigma_{k}(u_0) + A_k^{\sigma}(u_0,...,u_{k-1}),
\end{equation}
where $A_k^{\sigma}(u_0,...,u_{k-1})$ is a $d\times d$ matrix which
depends only on the indicated variables. \\
We shall now apply the formulae we have obtained for $[\beta \left(
u(\varepsilon) \right)]_k$  and $[\sigma_{\varepsilon} \left(
u(\varepsilon) \right)]_k,\,k\,\in\,\N_0,$ to the case where
$u(\varepsilon)$ is replaced by the pathwise solution
$u_{\varepsilon}(s)$ of (\ref{eqn:1}), assumed to exist and to have
an asymptotic expansion in $\varepsilon$ of the form (\ref{eqn:2}).
By matching coefficients of the same order $k$ on both sides of
(\ref{eqn:1}), i.e. $u_k(t)$ resp. $\int_{0}^t\,[\beta \left(
u_{\varepsilon}(s) \right)]_k\,ds$ and
$\int_{0}^t\,[\sigma_{\varepsilon} \left( u_{\varepsilon}(s)
\right)]_k\,\eta(ds),$ we get the following Proposition:

\begin{prop}\label{haffa}
Let us assume that the coefficient $\sigma_{\varepsilon}$ is ${\cal
C}^{M+1}$ in $\varepsilon\,\in\,[0,\,\varepsilon_0)$ for some
$\varepsilon_0\,>0,\,$ in the sense that (\ref{abd3}) holds.
Moreover assume that $\beta \in {\cal
C}^{p+1}(\R^d),\,\sigma_{\varepsilon} \,\in\,{\cal C}^{s+1}(\R^d),$
for any $\varepsilon\,\in\,[0,\varepsilon_0),$ for some
$p\,\in\,\N_0,\,s\,\in\,\N_0.$\\
The stochastic equation (\ref{eqn:1}) has a pathwise solution
$u_{\varepsilon}$ for all $t\,\in\,[0,T],\,T> 0$ and the solution
$u_{\varepsilon}(t)$ is ${\cal C}^{m+1},\, $for some $m\,\in\,\N$,
in $\varepsilon\,\in\,[0,\,\varepsilon_0),\,$ i.e. (\ref{eqn:2})
holds. Then the expansion coefficients $u_k$ of (\ref{eqn:1}) in
(\ref{eqn:2}) satisfy the following equations:
\begin{equation}\label{zero}
   {u}_0(t) = u^0\,+ \int_0^t \beta(u_0)\,ds + \int_0^t\,\sigma_0(u_0)\eta(ds) \text{;}
\end{equation}
\begin{eqnarray}\label{first}
u_1(t)&=&\sum_{i=1}^d\,\Big[\int_0^t\,\partial_i\,\beta(u_0)\,u_{1,i}ds+\,\int_0^t\,\,\partial_i\,\sigma_{0}
\left( u_0 \right)\,u_{1,i}\,\eta(ds)\Big]+\int_0^t\,\sigma_{1}
\left( u_0
\right)\eta(ds)\nonumber\\&=&\int_0^t\,[\beta_{\varepsilon} \left(
u_{\varepsilon}(s) \right)]_1\,ds +\,\int_0^t\,[\sigma_{\varepsilon}
\left( u_{\varepsilon}(s) \right)]_1\, \eta(ds),
\end{eqnarray}
\begin{eqnarray}\label{second}
u_2(t)&=&
\ds_{i=1}^d\,\Big[\int_0^t\,\partial_i\,\beta(u_0)\,u_{2,i}\,ds+\,{1
\over
{2!}}\,\int_0^t\,{\ds_{i=i'=1}^d}^{\star}\,\partial_i\partial_{i'}\,\beta(u_0)\,u_{1,
i}\,u_{1,i'}ds\Big]\nonumber\\&+&\ds_{i=1}^d\,\int_0^t\,\Big[\partial_i\,\sigma_{0}(u_0)\,u_{2,i}\,+\,{1
\over
{2!}}\,{\ds_{i,i=1}^d}^{\star}\,\,\partial_i\partial_{i'}\,\sigma_{0}(u_0)\,u_{1,
i}\,u_{1,i'}\nonumber\\&+&\,\partial_i\,\sigma_1(u_0)\,u_{1,i}\,\Big]\eta(ds)+\int_0^t\,\sigma_2(u_0)\,\eta(ds)
\nonumber\\&=&\int_0^t\,[\beta \left( u_{\varepsilon}(s)
\right)]_2\,ds +\,\int_0^t\,[\sigma_{\varepsilon} \left(
u_{\varepsilon}(s) \right)]_2\, \eta(ds),
\end{eqnarray}
and for all $1\leq\,k\leq \min(Mp,\,Ns)$:
\begin{equation}\label{hiddi}
{u}_{k}(t) = \int_0^t\,[\beta \left( u_{\varepsilon}(s)
\right)]_k\,ds +\,\int_0^t\,[\sigma_{\varepsilon} \left(
u_{\varepsilon}(s) \right)]_k\, \eta(ds),
\end{equation}
where $[\beta \left( u_{\varepsilon}(s) \right)]_k$ and
$\,[\sigma_{\varepsilon} \left( u_{\varepsilon}(s) \right)]_k$ are
given in (\ref{hiii}), (by replacing $f$ by $\beta$), resp. in
(\ref{mama}) ($u(\varepsilon)$ being replaced by
$u_{\varepsilon}(s).$)
\end{prop}
\begin{proof}
The proof was already carried through before the Proposition.
\end{proof}
\begin{rem}\label{RHOUMA}
    \quad\vspace{-\baselineskip}
    \begin{enumerate}
    \item  For the existence and uniqueness of
    solutions of (\ref{eqn:1}) see sect. 3.
    \item If $\beta (\, x \,) = Ax + b$, ( with $b\,\in\,\R^d$ independent of $x$  and $A$ a $d\times d$-matrix independent of x) and
    $\sigma_{\varepsilon}(x) = \sigma_0 + \varepsilon \sigma_1 (x)
        $,\, with\,
         $\sigma_0=c$ and $(\sigma_1(x))_{l,l'}=\lambda_{l,l'}\,x_{l'},\,x\,\in\,\R^d,$ with $c,\,\lambda$ constant $\, d\times d-$ matrices,
         then $\sigma_i = 0$, $\forall\, i \geq 2$ and $\partial_{i_1}
\cdot\cdot\cdot\partial_{i_n}\beta = 0,\, \forall\, n  \geq\, 2$.
          Moreover, $(\partial_i\beta)_{l}(x)=\,{\partial\,\over\,{\partial
          {x_i}}}\ds_{k=1}^d\,A_{lk}\,x_k=A_{li},\,\,l=1,...,d $ and
          thus
          \begin{equation}\label{kh}
\ds_{k=1}^d\,\partial_i\beta(u_0)\,u_{k,i}=A\,u_k,\,k\,\in\,\N.
          \end{equation}
Moreover,
\begin{equation}\label{kk}
\partial_{i_1}
\cdot\cdot\cdot\partial_{i_k}\,\sigma_{0}(x)=0,\,\forall\,k\,\in\,\N,\,x\,\in\,\R^d,\,\,(\partial_i\sigma_1(x))_{l,l'}=\lambda_{l,l'}\,\delta_{l,l'},
\end{equation}
with $\delta_{l,l'}$ the Kronecker symbol, $l,l',i=1,...,d.$\\
Furthermore $\partial_{i_1}
\cdot\cdot\cdot\partial_{i_k}\,\sigma_{1}(x)=0,\,\forall\,k\,\geq2,\,\partial_{i_1}
\cdot\cdot\cdot\partial_{i_k}\,\sigma_{j}(x)=0,\,\forall\,j\,\geq2,\,k\,\in\,\N_0.$
Hence from (\ref{haffa}) we get:
\begin{equation}\label{kkk}
 u_0(t) = u^0+ \int_0^t\,Au_0(s)\,ds + b\,t + c\eta(t),
\end{equation}
and
\begin{equation}
({u}_{k}(t))_l=\int_0^t\,(A\,u_k(s))_l\,ds+\,\ds_{l'=1}^d\,\lambda_{l,l'}\,\int_0^t\,({u}_{k-1}(s))_{l'}\eta_{l'}(ds),\,\,l=1,...,d,\,k\,\in\,\N.
\end{equation}
\item $\beta$ and $\sigma_{\varepsilon}$ is as in 2., however with $\sigma_0=0$ replaced by $\sigma_0(x)=\Pi\,x,\,\Pi$ a constant
$d\times d-$matrix, then (\ref{kh}) holds, the first equation in
(\ref{kk}) is for $k=1$, replaced by
$(\partial_{i_1}\,\Pi\,x)_l=\Pi_{l,{i_1}}$, thus (\ref{kkk}) is
replaced by
\begin{equation}\label{kkkk}
 u_0(t) = u^0+ \int_0^t\,Au_0(s)\,ds + b\,t + \Pi\,\eta(t),
\end{equation}
and
\begin{equation}
({u}_{1}(t))_l=A\,\int_0^t\,(\,u_1(s))_l\,ds+\,\ds_{l'=1}^d\,\lambda_{l,l'}\,\int_0^t\,({u}_{0}(s))_{l'}\eta_{l'}(ds)\,\,+
\,\,\ds_{i=1}^d\,\Pi_{l,i}\,u_{1,\,i}\,\eta_i(ds),
\end{equation}
for $k\geq\,2,\, l=1,...,d:$
\begin{equation}
({u}_{k}(t))_l=\int_0^t\,(A\,u_k(s))_l\,ds+\,\ds_{l'=1}^d\,\lambda_{l,l'}\,\int_0^t\,({u}_{k-1}(s))_{l'}\eta_{l'}(ds)\,\,+
\,\,\ds_{i=1}^d\,\Pi_{l,i}\,u_{k,\,i}\,\eta_i(ds).
\end{equation}

\item If $\beta(x)=Ax+F(x),\,F\,\in {\cal C}^{p+1}(\R^d)$
and $\sigma_0=0,\,\sigma_1(x)=\,\Lambda,$ with $\Lambda$ a constant
matrix, so that (\ref{eqn:1}) has additive noise, then
\begin{equation}
u_0(t) = u^0+ \int_0^t\,Au_0(s)\,ds + \int_0^t\,F(u_0)\,ds,
\end{equation}
and
\begin{equation}
u_1(t) = \int_0^t\,Au_1\,ds +
\int_0^t\,\partial_i\,F(u_0)\,u_{1,i}ds+\Lambda\,\eta(t).
\end{equation}
The $u_k(t),\,k\geq\,2$ are given by linear nonhomogeneous
stochastic equations with random coefficients, depending only on the
$u_0,\,...,\,u_{k-1}$, without any external noise term. The
expansion is then a particular case of the one explicitly given in
\cite{ALDIBOU} (in the case where the Hilbert space is $\R^d$.)
        \item Equation (\ref{zero}) is of the same type as equation (\ref{eqn:1}) with $\varepsilon=0$. Only for
        $\sigma_0\equiv\,0$ we have a purely deterministic
        equation. The expansion in powers of $\varepsilon$ of the solution of (\ref{eqn:1}) is really useful whenever (\ref{zero}) can be better handled than the original
        equation (\ref{eqn:1}), which happens whenever $\sigma_0$ has a simpler dependence on $x$ than $\sigma_{\varepsilon}$ itself. See sect. 5, for more details.\\
Let us also underline that the equations (\ref{first})-(\ref{hiddi})
for the $u_k(t)$ are linear nonhomogeneous, with random coefficients
involving only $u_0,...,u_{k-1},$ hence to be solved recursively.
\item If the coefficient $\beta$ in (\ref{eqn:1}) depends itself on
$\varepsilon\,\in\,[0,\,\varepsilon_0)$ and is in ${\cal C}^{\tilde{
M}+1}$ as a function of $\varepsilon$, thus has an expansion
$\beta(x)=\sum_{i=1}^M\,\beta_i(x)\,\varepsilon^i\,+\,R^{\beta}_{\tilde{M}}(\varepsilon,
x),\,\forall\,x\,\in\,\R^d,$ then the expansion
(\ref{f0})-(\ref{hiii}) with $f$ replaced by any component of
$\beta$ has to be replaced by (\ref{s0})-(\ref{mama}), with the
matrix elements of $\sigma_{\varepsilon}$ replaced by the components
of
$\beta_{\varepsilon},\,i.e.\,\,[\beta_{\varepsilon}(u(\varepsilon))]_0=\beta_0(u_0),\,[\beta_{\varepsilon}(u(\varepsilon))]_1=\ds_{i=1}^d\,\partial_i\,\sigma_{0}
\left( u_0 \right)\,u_{1,i}+\sigma_{1} \left( u_0 \right)$ and
correspondingly for (\ref{s2}). (\ref{mama}) holds for $\beta$
replaced by $\beta_0$, whereas in the equations for the
$u_k(t),\,k\,\in\,\N$ we have to replace $[\beta \left(
u(s,\,\varepsilon) \right)]_k$ by $[\beta_{\varepsilon} \left(
u(s,\,\varepsilon) \right)]_k$, with
\begin{eqnarray}
[\beta_{\varepsilon} \left( u(s,\,\varepsilon)
\right)]_k&=&\ds_{i=1}^d\,\partial_i\,\beta_0(u_0)\,u_{k,i}\,+\,{1
\over {k!}}\,{\ds_{i_1,...,i_k=1}^d}^\star\,\partial_{i_1}
\cdot\cdot\cdot
\partial_{i_k}\,\beta_0(u_0)\,u_{1, i_1}u_{1, i_2}\cdot\cdot\cdot u_{1,
i_k} +\,\beta_k(u_0)\nonumber\\&+&A_k^{\beta}(u_0,...,u_{k-1})
\end{eqnarray}

\end{enumerate}
\end{rem}
\section{Existence and uniqueness results for the original SDE.}
\label{sec:1}

Let $\beta: \mathbb{R}^d \to \mathbb{R}^d$, $\beta = (\beta^1,\ldots,\beta^d)$, $\beta_i: \mathbb{R}^d \to \mathbb{R}$, $i=1,\ldots,d$, and let
 $\sigma = \left( \sigma_j^i \right)$, with $\sigma_j^i: \mathbb{R}^d \to \mathbb{R}, i,j=1,...,d.$  \\
We assume that $\beta$ is globally Lipschitz, i.e. $\left| \beta(x)
- \beta(y) \right| \le k_{\beta} \left| x-y \right|$, for all $x,y
\in \mathbb{R}^d,$ for some constant $k_{\beta} > 0$. \\
We also assume $\sigma_j^i$ are globally Lipschitz, i.e. $\left| \sigma_j^i(x) - \sigma_j^i(y) \right| \le k_{\sigma_j^i} \left| x-y \right|$,
for some constant $k_{\sigma_j^i}\,>0$ (independent of $x,y$) and all $x,y \in \mathbb{R}^d,\,i,j=1,...,d$. \\
Let $L(t)$ be a L\'{e}vy process on $\mathbb{R}^d$, without Gaussian
and deterministic component, i.e. with characteristic function:

\begin{equation}\label{BOB}
    E \left( e^{i \langle u,L(t) \rangle} \right) = e^{\int_{\mathbb{R}^d \setminus \{ 0 \}} \left( e^{i \langle u,y \rangle} - 1 - i \langle u, y \rangle\,\chi_{B}(y) \right) \nu (\mathrm{d}y)} \text{,}
\end{equation}

$u \in \mathbb{R}^d$, $B$ the unit ball in $\mathbb{R}^d$. $\nu$ is the intensity or L\'evy measure,
 satisfying\\ $\int_{\R^d\setminus \{0\}} (\mid y\mid^2\,\wedge\,1)\,\nu(dy)\,<\,\infty.$ For L\'evy processes see, eg., \cite{DA},\cite{MaRu}, \cite{KS} .\\
The following L\'{e}vy-It\^{o} decomposition holds (see,
e.g.,\cite{ABR1}, \cite{ABRW}, [\cite{DA}, p. 108--109],
\cite{MaRu}).

\begin{equation}
    L_t = \int_B x \tilde{N} (t,\mathrm{d}x) + \int_{\mathbb{R}^d \setminus B} x N(t,\mathrm{d}x) \text{,} \ t \geq
    0,
\end{equation}

with $N$ a Poisson random measure on $\mathbb{R}_{+} \times
(\mathbb{R}^d - \{ 0 \})$ (the Poisson random measure associated
with $\Delta Z_t := X_t - X_{t_{-}}$, i.e. $N([0,t) \times A) =  \{
0 \le s < t \vert \Delta Z_s \in A \}$, for each $t \geq 0$, $A \in
{\cal B}(\mathbb{R}^d \backslash \{ 0 \})$, $\tilde{N} (t,A) :=
N(t,A) - t \nu (A)$, for all $A \in {\cal B}(\mathbb{R}^d, \{ 0
\})$, $0 \in \overline{A}$). We have $\nu(A) = E(N(1,A))$; for each
$t>0$, $\omega \in \Omega$, $\tilde{N}(t, \, \cdot \,) (\omega)$ is
the compensated Poisson random measure (to ${N}(t, \, \cdot \,)
(\omega)$) on the Borel $\sigma-$algebra ${\cal B}(\mathbb{R}^d
\backslash \{ 0 \})$; $\tilde{N}(t,A)$, $t \geq 0$ is, in
particular,
 a martingale-valued measure. \\
 It is known that the solution $u(t)$ of (\ref{eqn:1})  with $\eta(ds)=\,dL_s+ b\,ds+\,dB_A(s),\,b\,\in\,\R^d,\,B_A$ a Brownian motion in $\R^d$ with covariance matrix $A$, can be
 identified with the solution $X_t$ of (\ref{Bacari})  see, e.g. \cite{DA}:

\begin{eqnarray}\label{Bacari}
    \mathrm{d} X_t(x) &=& \sigma_{\varepsilon}\,(X_{ t_{-}})\,b\,dt+\,\sigma_{\varepsilon}\,(X_{t_{-}})\,dB_A(t)+\, \beta (X_{t_{-}}) \mathrm{d}t + \int_{0 < \left| x \right| \le 1} \sigma_{\varepsilon}(X_{t_{-}}) \tilde{N} (\mathrm{d}t, \mathrm{d}x)
    \nonumber\\&+& \int_{\left| x \right| > 1} \sigma_{\varepsilon}(X_{t_{-}}) N(\mathrm{d}t, \mathrm{d}x)
\end{eqnarray}





The following theorem holds:

\begin{theo}
    If the coefficients $\beta,\sigma$ satisfy the above global Lipschitz and growth conditions and $\eta$ is as above then there exists a
    strong, c$\grave{a}$dl$\grave{a}$g, adapted solution of the SDE (\ref{eqn:1})
     or (\ref{Bacari}) and the
     solution is unique, for any initial condition $x_0 \in \mathbb{R}^d$.
\end{theo}

\begin{proof}
    This is a particular case of results given, e.g., in [\cite{GiSk}, pp. 237], \cite{Pro}, [\cite{WAK}, p. 231] and \cite{ABRM}.
\end{proof}


\begin{rem}
Other existence and uniqueness conditions are known, see, e.g.
\cite{}. Particularly the Lipschitz conditions can be relaxed to
local ones, with a condition of at most linear growth at infinity
see, e.g., \cite{GiSk}. This (and the previous result) also holds
for the non autonomous case where $\beta,\,\sigma$ have an
additional explicit measurable dependence on $t$ and all constants
entering the Lipschitz and growth conditions are uniform in $t$.
\end{rem}
\section{Discussion of the equations for the expansion coefficients}

In this section we shall provide solutions as explicit as possible
to the equations (\ref{hiddi}), for the expansion coefficients of
the solution of (\ref{eqn:1}) in powers of the small parameter
$\varepsilon.$ We first observe that (\ref{hiddi}) is a
nonhomogeneous linear equation in $u_k$ of the form:

\begin{eqnarray}\label{galia}
du_{k,l}(t)&=&[\tilde{F}_{k,l}(t)+\,\ds_{l'=1}^d\,\tilde{\gamma}_{k,l,l'}(t,u_0)(t)\,u_{k,l'}(t)]dt+\ds_{j=1}^d\,\tilde{G}_{k,l,j}(t,u_0(t),u_k(t))\,d\eta_j(t)\nonumber\\&+&
\ds_{l'=1}^d\tilde{g}_{k,l,l'}(t,u_0)\,d\eta_{l'}(t),\,\,\,\,\,\,\,
k\,\in\,\N,\,l=1,...,d,
\end{eqnarray}
with
\begin{equation}\label{SAM}\left\{\begin{array}{lllll}\tilde{F}_{k,l}(t):=[\beta_l(u(t,\varepsilon))]_k\,\,\,(with\,\,[.]_k\,\,given\,\,by\,\,(\ref{f0})-(\ref{hiii})\,\,\,,
for\,\,k\geq\,2),\,\,\tilde{F}_{0,l}(t)=\tilde{F}_{1,l}(t)):=0;&
\\\,\,\tilde{\gamma}_{k,l,l'}(t)=\partial_{l'}\beta_{l}(u_0),\,l=l'=1,...,d,\,k\,\in\,\N,
\tilde{\gamma}_{0,l,l'}(t):=0;&\\
\,\,\tilde{G}_{k,l,j}(t,u_0(t),u_k(t))=\ds_{i=1}^d\,
\partial_i\sigma_{0,l,j}(u_0(t))\,u_{k,i},&\\
\,\,and,\,\,\,\,\, for\,\,\,\,\, k \geq\,2,\,\,&\\
 \tilde{g}_{k,l,l'}(t)={1
\over {k!}}\,{\ds_{i_1,...,i_k=1}^d}^\star\,\partial_{i_1}
\cdot\cdot\cdot
\partial_{i_k}\,\sigma_{0}(u_0)\,u_{1, i_1}u_{1, i_2}\cdot\cdot\cdot u_{1,
i_k}+\sigma_{k}(u_0) + A_k^{\sigma}(u_0,...,u_{k-1})
\end{array}
\right.
\end{equation}
We observe that (\ref{SAM}) constitues a set of recursive equations,
where the $k-$th order equations, for $X_{k,l},\,l\,\in\,\{1,...,d
\}$, only involves the components of $X_k,$ in a linear way, with
random coefficients $f_k,\,g_k$ depending on the vectors
$X_0,...,X_{k-1},$ and with a random inhomogeneity depending on
$X_{k,l'},\,l'\neq l$. It is thus of the form
\begin{eqnarray}
dX_{k,l}(t)&=&
f_k(X_0,...,X_{k-1})\,X_{k,l}\,dt\,+\,g_k(X_0,...,X_{k-1})\,X_{k,l}\,\eta(dt)+h_k(X_0)\,\eta(dt)\,\nonumber\\&+&\,h_k(X_{0,l'},...,X_{k-1,l'})\,dt,
\,\,\,l'=1,...,d,l'\neq l, l=1,...,d,\,k=1,...,K.
\end{eqnarray}
Under Lipschitz assumptions and at most polynomial growth at
infinity  in the space variable for $\beta,\,\sigma_{\varepsilon}$
and their derivatives up to order $K,$ we can apply  methods similar
to the one used in \cite{ALDIBOU, ALEBOU} (in the infinite
dimensional case, however with additive noise) to show that
existence and
uniqueness of solutions holds. Also proofs can be adopted to cover our case starting from literature on the martingale method, see, e.g., \cite{NEG}.\\
Yet still in the latter case and even for $\eta$ a Brownian motion
no
"explicit" solutions are known.\\
In general, even in the $1-$dimensional case, it is difficult to
find explicit solutions. We already see that the equation for $u_1$
is a nonhomogeneous linear stochastic differential equation for
$u_1$ involving random coefficients and an inhomogeneity depending
on the
solution $u_0$, and the coefficients are in general non linear in $u_0$.\\
In the $d-$dimensional case where $\sigma_1(y)=\,a\,y$ and
$\sigma_0(y)=\,b\,y$ for some constant matrices $a,\,b\,\in\,\R.$
and $\beta(x)=\,c\,x+d,$ for some $c,\,d\,\in\,\R,\,x\,\in\,\R^d,$
then the linear equations for $u_0,\,u_1$ have constant coefficients
and it is easy from (\ref{hiddi}), (\ref{hiii}) and (\ref{mama}), to
see that also the equations for the $u_k,\,k\,\geq\,2$ are also of
this type. In this case, at least for $\eta=B$ a Brownian motion, we
can apply results on systems of linear equations with terms of at
most first order in the state variables, which are to be found,
e.g., in \cite{GAR} and \cite{ARN},
to find an explicit expansion for $u_k.$\\
We can thus apply to the discussion of (\ref{hiddi}) results on the
solution of linear deterministic resp. stochastic evolution
equations, according to the following proposition:
\begin{prop} \label{mamia}Consider a system of $K$ coupled linear stochastic
evolution equation with random coefficients, the coefficients of the
equation for the $k-$th component $k=1,...,K$ being only dependent
of the components of index $0,1,2,...,k-1.$ The equation for the
$l-$th component of the $k-$th vector, $X_{k,l}$ is of the form:
\begin{equation}\label{eee1}\displaystyle  \begin{array}{lll} dX_{k,l}(t)= & [F_{k,l}(t)+\,\ds_{l'=1}^d\,\gamma_{k,l,l'}(t)\,X_{k,l'}(t)]dt+
\,\ds_{j=1}^m\,G_{k,l,j}(t,\,X_k(t))\,\,dB_j(t)+
\,\ds_{l''=1}^d\,g_{k,l,l''}(t)\,dB_{l''}(t),\\
  \end{array}
 \end{equation}
 with all components of $\gamma_k,\,g_k$ independent of $X_k$, and
 $F_{k,l}$ as well as $G_{k,l,j}$ linear in the components $X_{k,l}$ of $X_k$ and independent of other state variables.\\
 All coefficients $F,\,\gamma,\,G,\,g$ are supposed to be Lipschitz
 and satisfy the linear growth conditions, with constant uniform in
 $t$. The explicit dependence of all coefficients on $t$ is supposed to be
 measurable.\\
The solution of (\ref{eee1}) is given by :

\begin{eqnarray}\label{hh}
X_{k,l}(t)&=&
\ds_{k',l'}\Phi_{k,l,k',l'}(t)\Big\{\ds_{k'',l''}\dint_0^t\,
\Phi_{k',l',k'',l''}^{-1}(s)[F_{k'',l''}(s)-\,
G_{k'',l',l''}(s)\,g_{k'',l',l''}(s)]ds\nonumber\\&+&\,\ds_{k'',l''}\dint_0^t\,\Phi_{k',l',k'',l''}^{-1}(s)\,\,g_{k'',l',l''}(s)\,dB_{l''}(s)
\Big\},
\end{eqnarray}
 where the summation being over $k',k''=1,...,K$ and
 $l',l''=1,...,d,$ for all $k=1,...,K,\,l=1,...,d.$ For $k=0,\,X_{0,l}(t)$ is the solution of (\ref{zero}).\\
 $\Phi$ is the fundamental $Kd\times Kd$ matrix of
the corresponding homogeneous equation, i.e. the equation
(\ref{eee1}) with $F=g=0$, normalized so that $\Phi(0)$ is the unit
matrix, and the integrals being understood in It$\hat{o}'$s sense.
\label{POP}
\end{prop}
\begin{proof}
    The proof uses It$\hat{o}'$s formula to identify $dX_t$ as given by the right hand side of (\ref{hh}) with the right hand side of (\ref{eee1}).
    The presence of $\Phi(t)$ is for similar reasons as in
    Lagrange's method for systems of ODEs, see \cite{GAR} and \cite{ARN}, to which we refer for details.
\end{proof}
\begin{rem}\label{ra1}
For $K,\,d=1$, the fundamental $Kd\times Kd$ matrix reduces to a
scalar $\Phi$. In this case we have (see e.g., [\cite{GAR}, p.113]):

\begin{equation}\label{gg}
\Phi(t)=\exp{\Big(\dint_0^t\,[\gamma(s)-{1 \over
2}{G}^2(s)]ds\,+\,\dint_0^t\,{G}(s)\,dB(s)\Big)},
\end{equation}
 In the case where $\eta$ is the
sum of $B$ and a jump component we have instead:
\begin{equation}\label{haw}
\Phi(t)=\exp{\Big(\dint_0^t\,[\gamma(s)-{1 \over
2}{G}^2(s)]ds\,+\,\dint_0^t\,{G}(s)\,dB(s)\Big)}\dprod_{0< s\leq
t}(1+\Delta\eta_J(s))e^{-{\Delta\eta_J(s)}},
\end{equation}
where $\Delta\eta_J(s):= \eta_J(s)-\eta_J(s^{-}),$ is the jump of
$\eta_J$ between $s^{-}$ and $s$. The product term is the
Dol$\acute{e}$ans-Dade exponential of a jump process, see, e.g.
[\cite{DA}, p.247].
\end{rem}
\begin{rem}
\begin{enumerate}
\item The corresponding results holds also in the deterministic case where $B$ is
replaced by a function of bounded variation $\eta$ hold.\\
The concrete expression for $\Phi$ changes, due to the fact that
there is no corrections term in the exponents as in the Brownian
motions exponential. For example instead of (\ref{gg}) we simply
have then
\begin{equation}\label{ggg}
\Phi(t)=\exp{\Big(\dint_0^t\,\gamma(s)\,ds\,+\,\dint_0^t\,{G}(s)\,\eta(ds)\Big)},
\end{equation}
the second integral being a Stieltjes one.\\
\item In the case where $\eta$ contains a nontrivial jump component,
we were not able to find in the literature a general result of this
type.\\
For particular cases, e.g., where $\eta$ has also a component of
jump type, see however e.g. \cite{LAN}, \cite{DYJ}. As to be
expected in the exponent appearing in (\ref{gg}) an additional
Dol$\acute{e}$ans-Dade term (stochastic exponential of L\'evy jump
process) appears, see, e.g., [\cite{DA}, p.247].
\end{enumerate}
\end{rem}
As we stated before Proposition \ref{mamia}, that Proposition can be
applied to the case where $\eta=B$ and $\beta(x)=Ax\,+b\,,
\sigma_1(x)=\lambda\,x,\,\sigma_0(x)=\Pi\,x$ for all
$x\,\in\,\R^d,\,b\,\in\,\R^d$. Here $A,\,\lambda,\,\Pi$ are the
constant
$d \times d$ matrices discussed in remark \ref{RHOUMA}, 3.\\
From proposition \ref{mamia} and (\ref{SAM}) we get the following:
\begin{prop}\label{lp}
Let $\beta(x)=Ax\,+b\,, \sigma_1(x)=\lambda\,x,\,\sigma_0(x)=\Pi\,x$
for all $x\,\in\,\R^d,\,b\,\in\,\R^d$. $A,\,\lambda,\,\Pi$ are
constant $d \times d$ matrices. Consider the solution of the
equation $du=\,\beta(u)\,dt+\,\sigma_{\varepsilon}(u)\,\eta(ds).$\\
Let $u_k$ be the expansion coefficients which satisfy the equations
in Proposition \ref{haffa}. Then the $l-$th component $u_{k,l}$ of
$u_k$ is given by
\begin{equation}
u_{k,l}(t)=\ds_{k',l'}\,\Phi_{k,l,k',l'}(t)\Big\{ -\ds_{k'',i}
\,\dint_0^t\,
\Phi_{k',l',k'',i}^{-1}(s)\,u_{k-1,i}(s)\,\Big(\lambda_{l,i}\,u_{k'',i}(s)\,ds+\lambda_{l',i}\,dB_i(s)\,\Big)\,\Big\}.
\end{equation}
$\Phi$ is the fundamental matrix of the system
\begin{equation}
du_{k,l}=(Au_k)_l(s)\,ds\,+\,\ds_{i=1}^d\,\Pi_{l,i}\,u_{k,i}\,dB_i(s).
\end{equation}
\end{prop}

\section{The asymptotic character of the expansion}
In this section we shall prove the asymptotic character of the
expansion of the solution $u_{\varepsilon}$ of (\ref{eqn:1}) in
powers of $\varepsilon$, under the hypothesis of sect. 3. By so doing we provide more details and precision to a method sketched (for $d=1.$)\\
Let $u_j,\,j=1,...,k$ be the coefficients in a (first heuristic )
asymptotic expansion of $u_{\varepsilon}(t)$ in powers of
$u_{\varepsilon}$, the equations of which we discussed in sect. 4.\\
Let us study
\begin{equation}
R_k(t,
\varepsilon):=\,\varepsilon^{-(k+1)}[u_{\varepsilon}(t)-\ds_{j=0}^k\,\varepsilon^j\,u_j(t)],\,\,t\geq\,0,\,\varepsilon\,\in\,[0,\,\varepsilon_0].
\end{equation}
We have, using that $u_{\varepsilon}(t)$ solves (\ref{eqn:1}):
\begin{equation}
R_k(t,
\varepsilon)=\,\varepsilon^{-(k+1)}[\dint_0^t\,\beta(u_{\varepsilon}(s))\,ds+\dint_0^t\,\sigma_{\varepsilon}(u_{\varepsilon}(s))\,ds
-\ds_{j=0}^k\,\varepsilon^j\,u_j(t)],\,\,t\geq\,0,\,\varepsilon\,\in\,[0,\,\varepsilon_0].
\end{equation}
Let, for $\underline{y}=(y,y_0,...,y_k)\,\in\,\R^{(k+1)d}:$
\begin{equation}
A^{\beta}_{k+1}(\underline{y}):=\,\varepsilon^{-(k+1)}\Big[\beta(\ds_{j=0}^k\,\varepsilon^j\,y_j\,+\,\varepsilon^{k+1}\,y)-\ds_{j=0}^k\,\varepsilon^j\,\beta_j(y_0,...,y_j)],
\end{equation}
where $\beta_j(y_0,...,y_j)$ is the coefficient of the $j-$th power
in $\varepsilon$ of
$\beta(\ds_{l=0}^{M_\beta}\,\varepsilon^l\,y_l),\,M_\beta\,\geq\,k.$\\
Define correspondingly
\begin{equation}
A^{\sigma}_{k+1}(\underline{y}):=\,\varepsilon^{-(k+1)}\Big[\sigma_{\varepsilon}(\ds_{j=0}^k\,\varepsilon^j\,y_j\,+\,\varepsilon^{k+1}\,y)-\ds_{j=0}^k\,\varepsilon^j\,\sigma_j(y_0,...,y_j)],
\end{equation}
with $\sigma_j(y_0,...,y_j)$ the coefficient of the $j-$th power in
$\varepsilon$ of
$\sigma_{\varepsilon}(\ds_{l=0}^{M_{\sigma}}\,\varepsilon^l\,y_l),\,M_{\sigma}\,\geq\,k.$\\
By Taylor theorem applied to the expansion of
$A^{\beta}_{k+1}(\underline{y})$ as a function of $\underline{y}$
around $\underline{y}_{0}=(y_0,0,...,0)$, we have
\begin{equation}
\mid A^{\beta}_{k+1}(\underline{y})\mid\,\leq\,\mid
\beta^{(k+1)}(y_{\varepsilon}^{\star})\mid
\end{equation}
for some $y_{\varepsilon}^{\star}\,\in\,\R^d, $ depending on
$\varepsilon$ and $\underline{y}_0$ but such that
\begin{equation}
\dsup_{\varepsilon, \underline{y}}\,\mid
\beta^{(k+1)}(y_{\varepsilon}^{\star})\mid\,\leq\,C_{k+1},
\end{equation}
with $C_{k+1}$ a constant, provided
$\beta^{(k+1)}(y),\,y\,\in\,\R^d,$ is uniformly bounded, which we
assume. Hence
\begin{equation}
\dsup_{ \underline{y}}\,\mid
A^{\beta}_{k+1}(\underline{y})\mid\,\leq\,C_{k+1}.
\end{equation}

 Similarly, assuming $\sigma_{\varepsilon}^{(k+1)}(y)$ is
uniformly bounded with respect to $\varepsilon$ and $y$ in matrix
norm, we have
\begin{equation}
\dsup_{\varepsilon, \underline{y}}\,\parallel
A^{\sigma}_{k+1}(\underline{y})\parallel\,\leq\,\tilde{C}_{k+1},
\end{equation}
for some constant $\tilde{C}_{k+1}$.\\ Substituting
$y_j=u_j(s),\,j=0,...,k,\,y=R_k(t,\varepsilon)$ into
$A^{\beta}_{k+1}(\underline{y})$ resp.
$A^{\sigma}_{k+1}(\underline{y})$ we have that the $\sup$\,over
$\varepsilon$ and $\underline{y}$ of
$A^{\beta}_{k+1}(\underline{y})$ resp.
$A^{\sigma}_{k+1}(\underline{y})$ is bounded by $C_{k+1}$ resp.
$\tilde{C}_{k+1}.$ This implies:
\begin{equation}
\dsup_{s\,\in\,[0,t],\,\varepsilon}\,\mid
A^{\beta}_{k+1}(u_0(s),...,u_k(s),R_k(s,\varepsilon)
)\mid\,\leq\,C_{k+1}
\end{equation}
and
\begin{equation}
\dsup_{s\,\in\,[0,t],\,\varepsilon}\,\parallel
A^{\sigma}_{k+1}(u_0(s),...,u_k(s),R_k(s,\varepsilon)
)\parallel\,\leq\,\tilde{C}_{k+1}.
\end{equation}
From this we get that
\begin{equation}
\dsup_{s\,\in\,[0,t],\,\varepsilon}\,\mid
A^{\beta}_{k+1}(u_0(s),...,u_k(s),R_k(s,\varepsilon)
)\mid^p\,\leq\,C^p_{k+1}
\end{equation}
and correspondingly for $A^{\sigma}_{k+1},$ for any $1\,\leq
p\,\leq\,\infty.$ Thus
\begin{equation}
\mathbb{E}\big(\dsup_{s\,\in\,[0,t],\,\varepsilon}\,\mid
A^{\beta}_{k+1}(u_0(s),...,u_k(s),R_k(s,\varepsilon) )\mid^p
\Big)\,\leq\,C^p_{k+1}
\end{equation}
and correspondingly for $A^{\sigma}_{k+1}.$ From this it follows
that $\dsup_{s\,\in\,[0,t],\,\varepsilon}\,\mid
A^{\beta}_{k+1}(u_0(s),...,u_k(s),R_k(s,\varepsilon) )\mid^p$
converges as $\varepsilon\,\downarrow\,0$ in $L^p(p)$, and
correspondingly for $A^{\sigma}_{k+1}.$ Hence there is a subsequence
$\varepsilon_l\,\downarrow\,0$ as $l\,\longrightarrow\,\infty$, s.t:
\begin{equation}
\dsup_{s\,\in\,[0,t]}\,\mid
A^{\beta}_{k+1}(u_0(s),...,u_k(s),R_k(s,\varepsilon_l) )\mid
\end{equation}
converges stochastically as $l\,\longrightarrow\,\infty$, and
correspondingly for $\beta$ replaced by $\sigma.$\\
Taking a common subsequence $\varepsilon_{l'}$ we have then that
both
\begin{equation}
\dsup_{s\,\in\,[0,t]}\,(
A^{\beta}_{k+1}(u_0(s),...,u_k(s),R_k(s,\varepsilon_l) ))
\end{equation}
and
\begin{equation}
\dsup_{s\,\in\,[0,t]}\,(
A^{\sigma}_{k+1}(u_0(s),...,u_k(s),R_k(s,\varepsilon_l) ))
\end{equation}
converge stochastically for
$\varepsilon_{l'}\,\downarrow\,0,\,\,\,l'\,\longrightarrow\,\infty$,
to limits $\tilde{A}^{\beta}_{k+1}$ resp.
$\tilde{A}^{\sigma}_{k+1}$.\\
But $R_k(t,\varepsilon)$ satisfies (by construction of the
$u_0,...,u_k$ and the definition of $A^{\beta}_{k+1}$) the
stochastic differential equation
\begin{equation}\label{BOOO}
dR_k(t,\varepsilon)=A^{\beta}_{k+1}(u_0(t),...,u_k(t),R_k(s,\varepsilon))\,dt+\,A^{\sigma}_{k+1}(u_0(t),...,u_k(t),R_k(t,\varepsilon))\,\eta(ds),\,\,
for\,\,any\,\,\varepsilon\,\,\in\,[0,\varepsilon_0].
\end{equation}
By a well known result for SDE's, see, e.g, \cite{GiSk}, we have
that, the solution $R_k(t,\varepsilon_{l'})$ of (\ref{BOOO}),  has a
finite limit, call it $R_k(t,\,0)$, in the
P-stoch$\dsup_{s\,\in\,[0,t]}\,$ sense, as
$\varepsilon_{l'}\,\downarrow\,0, l'\,\longrightarrow\,\infty$.\\
But by the definition of $R_k(t,\varepsilon)$ this means that
\begin{equation}
P-stoch\dsup_{s\,\in\,[0,t]}\,
\dlim_{l'\,\longrightarrow\,\infty}\,\varepsilon^{k+1}_{l'}\,R_k(t,\,\varepsilon_{l'})=0,\,
\end{equation}
i.e. the expansion
\begin{equation}
u_{\varepsilon_{l'}}(t)=\ds_{j=0}^k\,\varepsilon_{l'}^j\,u_j(t)+\,\varepsilon_{l'}^{k+1}\,R_k(t,\,\varepsilon_{l'}),
\end{equation}
is asymptotic in the P-stoch$\dsup_{s\,\in\,[0,t]}\,$ sense. Hence
we have proven the following

\begin{theo}
Let $u_{\varepsilon}(t),\,
\varepsilon\,[0,\varepsilon_0],\,\varepsilon_0\,>0,$ be the solution
of (\ref{eqn:1}). Assume $\beta$ and $\sigma_{\varepsilon}$ are
${\cal C}^{k+1}$ in the space variables and $\sigma_{\varepsilon}$
is ${\cal C}^M$ in $\varepsilon,$ for some $M\,\geq\,k+1$.\\
Assume also that $\beta$ and $\sigma_{\varepsilon}$ are such that
the $k+1-$derivatives of $\beta$ and $\sigma_{\varepsilon}$ are
uniformly bounded in the $\mid\,.\,\mid$ resp.
$\parallel\,.\,\parallel-$norm. Then there is a sequence
$\varepsilon\,\downarrow\,0,$ such that
\begin{equation}
u_{\varepsilon}(t)=\ds_{j=0}^k\,\varepsilon^j\,u_j(t)+\,\varepsilon^{k+1}\,R_k(t,\,\varepsilon),
\end{equation}
with
\begin{equation}
P-stoch\dsup_{s\,\in\,[0,t]}\,\mid\,R_k(t,\,\varepsilon)\,\mid\,\leq\,\varepsilon^{k+1}\,C_{k+1},
\end{equation}
where $C_{k+1}$ is a constant. The $u_j$ satisfy the equations
described in sect.3 and further in sect.4
\end{theo}
\begin{rem}
In the case of additive noise, with $\beta$ the sum of a linear part
and a polynomially bounded non linear part, we could adopt methods
developed in \cite{Ma}, \cite{AlDiPMa}, \cite{ALEBOU} in the
infinite dimensional case to our case, avoiding thus the
restrictions on growth at infinity of $\beta$ we assumed in sect.3.
See also, e.g,. \cite{Ku}, \cite{NEG}, \cite{WILL}.
\end{rem}

\section{A remark on some applications}
Heuristic asymptotic expansions in  small parameters, to a certain
order and more often without any proof of their asymptotic character
(because of lack of suitable estimates on the remainders) appear
often in the literature. E.g, in neurobiology, stochastic models of
the Fitz Hugh Nagumo type without space dependence have been
discussed extensively, at least with additive Gaussian noise.
Our method can be applied to them. Examples are discussed basically with additive noise, e.g., in \cite{AlDiP}, \cite{Tu2}.\\
Another area where we find examples is mathematical finance. If we
take
$\sigma_0=0,\,\sigma_1(x)=\tilde{\sigma}\,x,\,\tilde{\sigma}\,>0,\,\beta(x)=\,r\,x,\,x\,\in\,\R,$
i.e, we take the model of example2 in remark \ref{RHOUMA} , then
$u_{\varepsilon}$ satisfies the equation of a Black-Scholes model
with volatility parameter $\varepsilon\,\tilde{\sigma}$ and our
expansion is then a small volatility expansion, see also \cite{LUT}.
If we take instead
$\sigma_0(x)=\tilde{\sigma}\,x,\,\sigma_i(x)\,\neq\,0$ for some
$i\,\in\,\N,\,\beta(x)=rx,\,x\,\in\,\R,$ then we have a stochastic
volatility model with leading order given by the Black-Scholes
solution and the expansions
yields corrections around the Black-Scholes model.\\
Similar applications can be given to the multidimensional
Black-Scholes model, see, e.g, \cite{LUT}, \cite{UY} and the
AS-model for interacting assets discussed
in \cite{ALST}, \cite{LUT}.\\
For other applications in this area see also \cite{IP}.\\\\

\noindent{\bf Acknowledgements:} This work was supported by the
Department of Mathematics and Statistics at King Fahd University of
Petroleum and Minerals under the summer project term 123 and hosted
at the University of Bonn, Germany.\\ The authors gratefully
acknowledge this support.\\
We are also grateful to Luca Di Persio and Elisa Mastrogiacomo for
stimulating discussions and collaborations on related problems.

\vskip4cm
\begin{flushleft}
\footnotesize
{\it S. Albeverio}  \\
 Dept. Appl. Mathematics, University of Bonn,\\
HCM, BiBoS, IZKS; KFUPM (Chair Professorship) , \\
\medskip
{\it Boubaker Smii} \\
Dept. Mathematics, King Fahd University of Petroleum and Minerals, \\
Dhahran 31261, Saudi Arabia\\

\medskip
 E-mail: albeverio@uni-bonn.de \\
\hspace{1,1cm} boubaker@kfupm.edu.sa\\

 \end{flushleft}


\begin{thebibliography}{99}
\bibliographystyle{alpha}
\footnotesize
\bibitem{A}{\sc
    Albeverio. S.}
     Wiener and {F}eynman-path integrals and their applications,
  {\it Proceedings of the {N}orbert {W}iener {C}entenary {C}ongress,
              1994 ({E}ast {L}ansing, {MI}, 1994),
   Proc. Sympos. Appl. Math.},
    {\bf 52},
 Amer. Math. Soc., Providence, RI, 1997, pp. 153--194.
\bibitem{AlDiP} { \sc Albeverio. S and  Di Persio. L.} {\it Some
stochastic dynamical models in neurobiology: recent developments.}
Eur. Comm. Math. Ther. Bio. vol. 14,  pp.44-53. (2011)
\bibitem{ADPM}{
    \sc Albeverio. S, Di Persio. L and Mastrogiacomo. E.}
    {\it  Invariant measures for stochastic differential equations on
networks.} Tohoku Math. J., vol. 63,  pp.877-898 (2011).
\bibitem{AlDiPMa} {\sc Albeverio. S, Di Persio. L and Mastrogiacomo. E.}{\it Small noise asymtotic expansions for
stochastic PDE's I. The case of a dissipative polynomially bounded
nonlinearity.} Tohoku. Math. J, 63(2011), 877-898.
\bibitem{ALDIBOU} {\sc Albeverio. S, Dipersio. L, Mastrogiacomo. E and Smii. B.} {\it Invariant measures for stochastic differential equations driven by L\'evy
noise.} Preprint (2013).
\bibitem{AFER}
    {\sc Albeverio. S AND Ferrario. B.} {\it Some methods of
    infinite dimensional analysis in hydrodynamic: recent progress
    and prospects. Lecture Notes in Math.  V 1942.} Springer,
    Berlin, 1-50 (2008).
\bibitem{AGoWu} {\sc  Albeverio. S, Gottschalk. H and J-L. Wu. } {\it Convoluted Generalized White noise, Schwinger Functions and their Analytic continuation to
 Wightman Functions.} Rev. Math. Phys, Vol. 8, No. 6, 763-817, (1996).
 \bibitem{AlGYo} {\sc Albeverio. S,  Gottschalk. H and Yoshida. M.W. } {\it System of classical particles in the Grand canonical ensemble, scaling limits and
 quantum field theory.}  Rev. Math. Phys, Vol. 17, No. 02, 175-226, (2005).
\bibitem{ALHKO} {\sc Albeverio. S, Hilbert. A and Kolokoltsov.} {\it
Uniform asymptotic bounds for the heat kernel and the trace of a
stochastic geodesic flow.}  Stochastics 84 no. 2-3, pp. 315–333,
(2012).
\bibitem{ALHKO1} {\sc Albeverio. S, Hilbert. A and Kolokoltsov.} {\it
Estimates uniform in time for the transition probability of
diffusions with small drift and for stochastically perturbed newton
equations,} J. Theor. Probab. 12, pp. 293-300, (1999).
 \bibitem{AHMA} {\sc Albeverio. S, Hoegh-Krohn. R and Mazzuchi. S.}
 {\it Mathematical theory of Feynamn path integrals.} 2nd.ed.
 Springer Berlin (2008).
 \bibitem{AL}{
   \sc Albeverio. S and Liang. S.}
   Asymptotic expansions for the {L}aplace approximations of sums
              of {B}anach space-valued random variables,
 {\it  Ann. Probab.}
   {\bf 33},
    (2005),
     pp. 300--336.
\bibitem{ALEBOU} {\sc Albeverio. S, Mastrogiacomo. E and Smii.
B.\,}{\it Small noise asymptotic expansions for stochastic PDE's
driven by dissipative nonlinearity and L\'evy noise.} Stoch. Proc.
Appl. 123(2013), 2084-2109.
\bibitem{ABRM} {\sc Albeverio. S, Mandrekar. V and R$\ddot{u}$diger B.}
{\it Existence of mild solutions for stochastic differential
equations and semilinear equations with non-Gaussian L\'evy noise.}
Stochastic Process. Appl. 119 (2009), no. 3, 835--863.
\bibitem{ALPS} {\sc Albeverio. S, Popovici. A and Steblovskaya. V.}
{\it A numerical analysis of the extended Black-Scholes model.} Int.
J. Theor. Appl. Finance 9(2006), no. 1, 69-89.
\bibitem{ABR1} {\sc Albeverio. S and R$\ddot{u}$diger B.} {\it Stochastic integrals and the L\'evy-It$\hat{o}$ decomposition theorem on separable Banach spaces.}
 Stoch. Anal. Appl. 23 (2005), no. 2, 217--253.
\bibitem{ABRW} {\sc Albeverio. S, R$\ddot{u}$diger. B and J-L. Wu.}
{\it Invariant measures and symmetry property of L\'evy type
operators.} Pot. Ana. 13 (2000), 147--168.
\bibitem{ARoSk}
    {\sc Albeverio. S, R$\ddot{o}$ckle. H and Steblovskaya. V.}
     Asymptotic expansions for {O}rnstein-{U}hlenbeck semigroups
              perturbed by potentials over {B}anach spaces,
{\it Stochastics Stochastics Rep.},
{\bf 69},
     (2000),
     pp. 195--238.
\bibitem{ASK}
    {\sc Albeverio. S and Steblovskaya. V.}
    Asymptotics of infinite-dimensional integrals with respect to
              smooth measures. {I},
  {\it Infin. Dimens. Anal. Quantum Probab. Relat. Top.},
     {\bf 2},
      (1999),
     pp. \ 529--556.
     \bibitem{ALST}{\sc Albeverio. S and Steblovskaya. V.} {\it Financial Market with Interacting Assets. Pricing Barrier
     Options.} Tr. Mat. Inst. Steklova. pp. 173-184 (2002).
     \bibitem{ALSW} {\sc Albeverio. S, Steblovskaya. V and Wallbaum.
     K.} {\it Valuation of equity-linked life insurance contracts
     using a model with interacting assets.} Stoch. Anal. Appl. 27,
     pp. 1077-1095 (2009).
\bibitem{AWZ}{\sc Albeverio. S ,  Wu. J-L and  Zhang. T. S.}
{\it Parabolic SPDEs driven by Poisson white noise.} Stoch. Proc.
Appl. 74 (1998), 21-36.
\bibitem{DA} {\sc  Applebaum. D .} {\it L\'evy Processes and
stochastic Calculus.} 2nd ed., Cambridge University Press 2009.
\bibitem{ARN} {\sc Arnold. L.} {\it Stochastic differential
equations. Theory and applications.} Wiley 1974.
\bibitem{ARNOL} {\sc Arnold. V.} {\it Mathematical methods of
classical mechanics.} Springer 1978.
 \bibitem{BN} {\sc Barndorff-Nielsen. O.E and  Shephard. N.}{\it Non-Gaussian Ornstein-Uhlenbeck-based models and some of their uses in financial economics.}
 J. R. Stat. Soc. Ser. B Stat. Methodol. 63, no. 2, 167-241 (2001).
 \bibitem{BIRKHOFF} {\sc Birkhoff, G.D.} {\it Dynamical systems,
 AMS, Providence.} 1927.
 \bibitem{BOGO} {\sc Bogoliubov. K.} {\it Introduction to nonlinear
 mechanics.} Kraus 1970.
\bibitem{BOG} {\sc  Bogoliubov. N and  Mitropolsky.Y. A.} {\it  Asymptotic Methods
in the Theory of Non-Linear Oscillations.} New York, Gordon and
Breach. 1961.
\bibitem{Bre} {\sc Breitung. K. W.} {\it Asymptotic approximations
for probability integrals. Lecture notes in Mathematics. Springer
1994.}
\bibitem{BRHA} {\sc Brze$\acute{z}$niak. Z and Hausenblas. E.} {\it
Uniqueness in law of the It$\hat{o}$ integral with respect to L\'evy
noise,} pp. 37-57 in Seminar Stoch. Anal., Random Fields and Appl.,
VI, Birkhauser, Basel (2011).
\bibitem{Costin} {\sc Costin. O and Tanveer. S.} {\it On the existence and uniqueness of solutions of nonlinear evolution systems of PDEs in $\R+\,\times\,\C^d,$
their asymptotic and Borel summability properties,} submitted,
available at http://www.math.ohio-state.edu/$\sim$
tanveer.
\bibitem{DYJ}{\sc Duan. J and Yan. J} {\it General matrix valued
inhomogeneous linear stochastic differential equations and
applications.} Stat. Prob. lett. pp. 2361-2365 (2008).
\bibitem{Eck} {\sc Eckhaus. V.} {\it  Asymptotic Analysis of Singular
Perturbations.} Elsevier Science Ltd. 1979.
\bibitem{EFR} {\sc Eckmann, J.-P., Epstein, H and Fr$\ddot{o}$hlich, J.} {\it Asymptotic
perturbation expansion for the S-matrix and the definition of time
ordered functions in relativistic quantum field models. Ann. Inst.
H. Poincar$\acute{e}$ }Sect. A (N.S.) 25 (1976/77), no. 1, 1–34.
\bibitem{CG} {\sc  Gardiner. C.} {\it Stochastic methods: A Handbook for the natural and social
sciences.} Springer-Verlag Berlin Heidelberg 2009.
\bibitem{GAR} {\sc Gard. T.} {\it Introduction to stochastic
differential equations.}Marcel Dekker Inc. 1988.
\bibitem{GiSk} {\sc Gikhman. I.I  and  Skorohod. A.V.} {\it Stochastic differential equations. Springer, Berlin 1972.}
\bibitem{MAG} {\sc Giaquinta. M and Modica. G.} {\it An introduction to functions of several
variables.} Birkh$\ddot{a}$user (2000).
\bibitem{GST} {\sc Gottschalk. H , Smii. B and  Thaler. H.} {\it
The Feynman graph representation of general convolution semigroups
and its applications to L\'evy statistics}. Journal of Bernoulli
Society,14(2), 322-351 (2008).
\bibitem{GS}{\sc Gottschalk. H  and Smii. B. }{\it How to determine the law of the solution to a SPDE driven by a L\'evy space-time noise,} {Journal of Mathematical Physics}
 v.43. 1-22 (2007).
\bibitem{GOS} {\sc Gottschalk. H, Smii. B and Ouerdiane. H.}
{\it Convolution calculus on white noise spaces and Feynman graph
representation of generalized renormalization flows.}
101-111pp.icmarp (2005).
\bibitem{WAK} {\sc Ikeda. M and Watanabe. S.} {\it Stochastic differential equations and diffusion processes. North-Holland/ Kodamsha, Amesterdam, Tokyo 1981.   }
\bibitem{IP} {\sc Imkeller. P, Pavlymkevich. I and Wetzel. T.} {\it
L\'evy driven diffusions with exponentially light jumps. Ann. Prob.
pp. 530-564 (2009).}
\bibitem{Kato} {\sc Kato. T.} {\it Perturbation Theory for Linear
Operators.} Springer Berlin Heidelberg. 2nd .Ed. 1980.
\bibitem{KSO} {\sc Kr$\acute{e}$e. P and Soize. C.} {\it Mathematics
of Random phenomena,} Reidel, Dordrecht, 1986.
\bibitem{Ku} {\sc Kurtz. Th.} {\it equivalence of stochastic
equations and martingale problems.} Stochastic analysis. pp.
113-130, D. Crisan(ed). Springer-Verlag Berlin Heidelberg. 2011.
\bibitem{LAN} {\sc L$\acute{E}$andre. R.} {\it Flot d'une $\acute{e}$quation
diff$\acute{e}$rentielle stochastique. S$\acute{e}$rie de Prob. XIX,
LN Math.,} Springer. pp. 271-278 (1985).
\bibitem{LUT} {\sc L$\ddot{u}$tkebohmert. E.} {\it An asymptotic expansion for
a Black-Scholes type model.} Bull. Sci. Math. pp. 661-685 (2004).
\bibitem{MaRu} {\sc Mandrekar. V and R$\ddot{u}$diger. B} {\it L\'evy Noises and Stochastic Integrals on Banach
Spaces.} Book in preparation.
\bibitem{Ma}
    {\sc Marcus. R.}
    Parabolic {I}t\^o equations with monotone nonlinearities,
   {\it J. Funct. Anal.},
     {\bf 29},
     (1978),
    no.\,3,
     pp. 275--286.
  \bibitem{Marc2}
     {\sc Marcus. R}.
     Parabolic {I}t\^o equations,
    {\it Trans. Amer. Math. Soc.},
   {\bf 198},
      (1974),
     pp. 177--190.
     \bibitem{Maslov} {\sc Maslov. V.} {\it Perturbation theory for the multidimensional Schr$\ddot{o}$dinger
     equation.} Uspekhi Mat. Nauk, 16:3(99) (1961), p. 217.
\bibitem{MBPR} {\sc Meyer-Brandis. T and Proske. F} {\it Explicit
representation of strong solutions of SDEs driven by infinite
dimensional L\'evy processes.} J. Theor. Prob. 23, 301-314 (2010).
\bibitem{MM} {\sc Milan. P and Miroslav. k.} {\it Stochastic equations for simple discrete models of epitaxial
growth.} Phys. Rev. E.  V.54 (4), (1996).
\bibitem{NEG} {\sc Negoro. I and Tsuchiya. M.} {\it Convergence and uniqueness theorems for Markov processes associated with Lévy
operators.} Prob. theo. and math. stat. Lecture Notes in Math. pp.
348-356. Springer, Berlin, 1988.
 \bibitem{Pro} {\sc Protter. P.} {\it Stochastic integrations and
 differential equations, Springer, 2005.}
 \bibitem{Reed}{\sc Reed. M and Simon. B.} {\it Methods of modern mathematical
 physics.} Vol 1-4. Academic Pr Inc. 1981.
 \bibitem{Sand} {\sc Sanders. J and Verhulst. F.} {\it Averaging Methods in Nonlinear Dynamical
 Systems.} Springer-Verlag 1985.
\bibitem{KS} {\sc  Sato. K.I.} {\it L\'evy processes and infinite divisible distributions.} Cambridge University Press, 1999.
 \bibitem{BS} {\sc Smii. B.}  {\it A Linked Cluster Theorem of the solution of the generalized Burger equation}, Applied Mathematical Sciences.(Ruse),
 vol. 6, no. 1, pp. 21-38, (2012).
\bibitem{Tu2}
    {\sc Tuckwell. H. C.}
    Introduction to theoretical neurobiology. {V}ol. 2, Nonlinear and stochastic theories,
    {\it Cambridge Studies in Mathematical Biology},
    {\bf 8},
Cambridge University Press,
   Cambridge,
     (1988),
     pp. xii+265.
\bibitem{UY} {\sc Uchida. M and Yoshida. N.}  {\it Asymptotic
expansion for small diffusions applied to option pricing.} Stat.
Int. Stoch. Proc. 7, 189-223 (2004).
\bibitem{Vas} {\sc  Vasil'eva.A and Uspekhi  N.}{ \it Asymptotic behaviour of solutions
to certain problems involving non-linear differential equations
containing a small parameter multiplying the highest derivatives.}
V. 18, 3(111),  Pages 15–86. (1963).
\bibitem{Wen} {\sc Wentzell F.} {\it Random perturbations of dynamical systems.} Springer Verlag. 2012.
\bibitem{WILL} {\sc Williams. D.R.E.} {\sc Pathwise solutions of
stochastic differential equations driven by L\'evy processes, Rev.
Math. Iberoamericana.} pp. 295-329 (2001).
\end{thebibliography}
\end{document}